\documentclass[12pt]{amsart}
\usepackage{amsthm}
\usepackage{amssymb}
\usepackage{amsmath}
\usepackage[shortlabels]{enumitem}
\usepackage{graphicx}
\usepackage{xypic}
\usepackage{array}
\usepackage{color}
\usepackage{lscape}

\newtheorem{theorem}{Theorem}[section]

\newtheorem{Lemma}{Lemma}[section]
\newtheorem{Cor}{Corollary}[section]
\theoremstyle{definition}

\newtheorem{example}{Example}[section]
\newtheorem{assumption}{Assumption}[section]
\theoremstyle{remark}
\newtheorem{claim}{Claim}[section]

\newtheorem{remark}{Remark}[section]

\newcommand{\brref}[1]{(\ref{#1})}

\newcommand{\milano}{Dipartimento di Matematica ``F. Enriques"
 \\ Universit\`a degli Studi di Milano \\ Via Saldini 50 \\ 20133
Milano, Italy}
\newcommand{\palermo}{Dipartimento di Matematica e Informatica
 \\ Universit\`a degli Studi di Palermo\\ Via Archirafi 34 \\ 90123
Palermo, Italy}

\newcommand{\Pin}[1]{{\mathbb P}^{#1}}

\newcommand{\rk}[1]{{\rm rk}\,(#1)}

\newcommand{\nbXi}{\mathbf{X}_i}

\newcommand{\Phicols}[2]{[\Phi^k_{h_1,h_2,h_3}]_#1^#2}

\newcommand{\T}{\mathcal{T}}
\newcommand{\tensor}{\otimes}
\newcommand{\bincof}[2]{\left(\begin{array}{c}
#1\\
#2
\end{array} \right)}
\newcommand{\Frank}[1]{F-$\rk{#1}$}
\title[F-rank and core of Grassman 3-tensors]{The Multilinear rank and Core of trifocal Grassmann tensors}
\author[M.Bertolini]{Marina Bertolini}
\email{marina.bertolini@unimi.it}
\address{\milano}

\author[G.M.Besana]{Gian Mario Besana}
\email{gbesana@depaul.edu}
\address{College of Computing and Digital Media \\ DePaul University \\ 243 South Wabash \\ Chicago IL, 60604 USA}

\author[G.Bini]{Gilberto Bini}
\email{gilberto.bini@unipa.it}
\address{\palermo}

\author[C.Turrini]{ Cristina Turrini}
\email{cristina.turrini@unimi.it}
\address{\milano}

\makeatletter
\renewcommand*\env@matrix[1][*\c@\maxMatrixCols c]{%
  \hskip -\arraycolsep
  \let\@ifnextchar\new@ifnextchar
  \array{#1}}
\makeatother

\begin{document}
\date{\today}
\begin{abstract}
 Closed formulas for the multilinear rank of trifocal Grassmann tensors are obtained. An alternative process to the standard HOSVD is introduced for the computation of the core of trifocal Grassmann tensors. Both of these results are obtained, under natural genericity conditions, leveraging the canonical form for these tensors, obtained by the same authors in a previous work.  A gallery of explicit examples is also included.
\end{abstract}

\maketitle

\section{Introduction}

Tensors, either as multidimensional arrays of data in applied settings or, more classically, as representations of multilinear applications among vector spaces, have recently attracted renewed attention: see, for instance, \cite{be-car-cat-gi-on}, \cite{LA}. Among the many fascinating and intricate problems in the study of tensors, the calculation of any of the various established notions of their rank marries theoretical interest and practical applications. In particular, the determination of the multilinear rank of a tensor, i.e. the ranks of all its flattening matrices, is part of the process needed to arrive at a {\it core} of a tensor, see Section \ref{hosvdgeneral}. Being able to successfully and efficiently compute a core of large tensors can be a crucial step for concrete applications in image processing, computer engineering, and data management.

The authors have been interested for a while in a class of tensors that arise naturally in computer vision. In the classical case of reconstruction of a three-dimensional static scene from two, three, or four two-dimensional images, these tensors are known as the fundamental matrix, the trifocal tensor, and the quadrifocal tensor, respectively, and have been studied extensively, see for example \cite{oe1, ah-st-th, al-to2, bebitu, Hart-Zi2,hey1,oe2}. In a more general setting, these tensors, called {\it Grassmann} tensors, were introduced by Hartley and Schaffalitzky, \cite{Hart-Schaf}, and were studied by three of the authors in several articles \cite{tubbMagic, tubbISVC08, tubbCHAPTER, tubbLAIA, tubbAMPA, tubbICCV07} as well as by two of the authors and other collaborators, \cite{notub}.

In \cite{BBBT1}, the authors leveraged the possibility of obtaining a canonical form for a general trifocal Grassmann tensor to compute its rank with a closed formula.

In this work we turn our attention to the multilinear rank of trifocal Grassmann tensors and to the related problem of computing their core. Under the same natural genericity assumption used in \cite{BBBT1}, see Assumption \ref{g.a.}, and similarly leveraging the resulting canonical form, in Section \ref{multilin} the multilinear rank of a trifocal Grassmann tensor is computed, with closed formulas as well.

A standard approach for the computation of a core $\mathcal{C}$ of a tensor $\T$ is to utilize the so-called {\it Tucker decomposition} \cite{tu}, often in the form of a higher order singular-value decomposition (HOSVD), \cite{del-dem-van, vann-vand-mee}. The Tucker decomposition combines the singular value decompositions  of all the flattenings $\T_i = U_i \Sigma_i W^*_i$ of the tensor in a multilinear multiplication $\mathcal{C} = (U_1^*, \dots, U^*_i,\dots) \cdot \T$ where $*$ denotes the adjoint matrix. Leveraging once again the canonical form for a trifocal Grassmann tensor, Section \ref{core} shows how to compute a core in a simpler alternative way. Properties of the canonical form of trifocal Grassmann tensors allow for a direct, immediate computation of its core. This canonical core can then appropriately be pulled back to produce a core for the original tensor. As part of this process, singular values of appropriate matrices still need to be computed, but the size of the matrices involved is, in general, significantly smaller than in the standard Tucker decomposition or HOSVD.

Examples of the explicit computation of the multilinear rank and the core are provided in Section \ref{examples}.

\section{Notation and Background Material}
\subsection{Notation}
Throughout this work we assume that the underlying field is the field $\mathbb{C}$ of complex numbers. Given a matrix $A$ with complex entries, $A^*$ denotes its adjoint matrix. For any positive integer $k,$ $I_k$ denotes the $k \times k$ identity matrix. A vector space $V$ of dimension $r$ over $\mathbb{C}$ is sometimes referred to as an $r$-space and $V^*$ denotes its dual, i.e. $V^* =$ Hom$_\mathbb{C}(V, \mathbb{C}).$
\subsection{Preliminaries on tensors} \label{introranks}
Notation and definitions of tensors
and their ranks (rank, multilinear rank or $F$-rank, P-rank) used in this work are relatively
standard in the literature. They are all contained in
\cite{LA} and briefly summarized below.

Given vector spaces $V_i, i = 1,\dots t,$  the {\it rank} of a tensor $\T \in V_1 \tensor V_2\tensor ... \tensor V_t,$ denoted by $R(\T),$ is the minimum number of decomposable tensors needed to write $\T$ as a sum. Recall that $R(\T)$ is invariant under changes of bases in the vector spaces $V_i $ (see, for example, \cite{LA}, Section 2.4 ).

This work focuses on a special class of trilinear tensors. 
For the convenience of the reader, and to fix our notation, it is useful to recall the explicit construction of the flattening matrices of a three dimensional tensor.

Let $V_1,V_2,V_3$ be vector spaces of dimension $n_1,n_2,n_3,$ with chosen bases $\{\alpha_i\}, \{\beta_j\}, \{\gamma_k\},$ respectively.

Let $\mathcal{T}=[T_{i,j,k}]\in V_1 \otimes V_2
\otimes V_3$. Interpreting $V_1 \otimes V_2 \otimes V_3$ as $V_1
\otimes (V_2 \otimes V_3)$, we get
\begin{equation}
\label{TA} \mathcal{T}=\sum_{i}\alpha_i \otimes
(\sum_{j,k}T_{i,j,k}(\beta_j \otimes \gamma_k))
\end{equation}
and the corresponding matrix, of size $n_1 \times (n_2 n_3)$, which
is the flattening $\mathcal{T}_1$, and has the following block
structure:

\begin{equation*}
{\tiny
\mathcal{T}_1=
\begin{bmatrix}[cccc|ccc|c|ccc]
  T_{1,1,1}&T_{1,2,1}  & \dots & T_{1,n_2,1}&T_{1,1,2} & \dots & T_{1,n_2,2}& \dots &T_{1,1,n_3}& \dots & T_{1,n_2,n_3} \\
   T_{2,1,1}&T_{2,2,1}  & \dots & T_{2,n_2,1}&T_{2,1,2}& \dots & T_{2,n_2,2}& \dots &T_{2,1,n_3}& \dots & T_{2,n_2,n_3}\\
 \ &\  & \vdots & \ & \   & \vdots &  \ & \dots & \  & \vdots & \\\
   T_{n_1,1,1}&T_{n_1,2,1}  & \dots & T_{n_1,n_2,1}&T_{n_1,1,2}& \dots & T_{n_,n_2,2}& \dots &T_{n_1,1,n_3}& \dots & T_{n_1,n_2,n_3}\\
\end{bmatrix}}
\end{equation*}
In the same way, paying attention to the cyclic nature of indices $i,j,k,$ one can define flattenings $\mathcal{T}_2$ and
$\mathcal{T}_3.$

One then defines the {\it mutilinear} rank (or  F-rank) of the tensor $\T$ as
F-$\rk\T =(\rk{\T_1},\rk{\T_2},\rk{\T_3}).$
\begin{remark}
\label{invfrk}
Let $M_r\in GL(n_r)$ be invertible matrices for $r =1,2,3.$ Let $\T_r$ be the $r$-th flattening of a tensor $\T$ as above. Then the F-$\rk\T$ is invariant under the left action of $GL(n_r)$ and the right action of $GL(n_s n_t)$ for $s,t \neq r.$ In particular the F-$\rk\T$ is invariant under the right multiplication of $M_s \tensor M_t \in GL(n_s n_t).$ 

\end{remark}







\subsection{Core of a Tensor}\label{hosvdgeneral}

Let $\mathcal{T}=[T_{i,j,k}]\in V_1 \otimes V_2 \otimes V_3$, where, as before, $V_1, V_2, V_3$ are vector spaces of dimension, respectively, $n_1,n_2,n_3$, with fixed bases and assume that F-$\rk{\T} = (r_1, r_2, r_3).$ Standard procedures in applications associate a {\em core tensor} ${\mathcal C}$ to ${\mathcal T}$. In this paper, following \cite{tu}, by a core tensor  of ${\mathcal T}$ we mean a tensor ${\mathcal C}$ that satisfies the following properties:
\begin{itemize}
    \item[1.] ${\mathcal C} \in Z_1 \otimes Z_2
\otimes Z_3$, where $Z_1, Z_2, Z_3$ are vector spaces of dimension, respectively, $r_1,r_2,r_3$;
\item[2.] there exist semi-orthogonal matrices $U_j$, i.e. $U_j^*U_j = I_{r_j}$, of size $n_j \times r_j$ for $j=1,2,3$ such that: \begin{itemize}
    \item [a.]the multilinear multiplication $(U^*_1,U^*_2,U^*_3) \cdot$ gives a map
$$(U^*_1,U^*_2,U^*_3) \cdot : V_1 \otimes V_2 \otimes V_3 \rightarrow Z_1 \otimes Z_2\otimes Z_3$$ with $$(U^*_1,U^*_2,U^*_3) \cdot \T = \mathcal{C};$$
\item[b.]the multilinear multiplication $(U_1,U_2,U_3) \cdot$ gives a map
$$(U_1,U_2,U_3) \cdot : Z_1 \otimes Z_2 \otimes Z_3 \rightarrow V_1 \otimes V_2\otimes V_3$$ with $(U_1,U_2,U_3) \cdot {\mathcal C} = \T.$ 
\end{itemize}  
\end{itemize}

We recall here the {\it{higher-order singular value decomposition}} (HOSVD) procedure which is the standard approach to the computation of a core of a tensor. It generalizes to tensors the standard (compact) singular value decomposition process for matrices. 

\noindent Let $\T$ be a tensor of order $3$ with flattening matrices $\T_1, \T_2, \T_3$ and \Frank{\T} $ = (r_1,r_2,r_3).$ Then
$\T_1$ is a $n_1 \times (n_2n_3)$ matrix and one can perform the (compact) SVD to $\T_1$:

$$ \T_1 = U_1 \Sigma_1 W_1^*,$$

\noindent where $\Sigma_1$ is a $r_1 \times r_1$ square diagonal matrix and where $U_1$ and $W_1$ are, respectively, $n_1 \times r_1$ and $(n_2n_3) \times r_1$ matrices such that $U_1^*U_1 = W_1^*W_1 = I_{r_1}$. Similarly, one can consider the SVD of $\T_2$ and $\T_3$, namely $ \T_2 = U_2 \Sigma_2 W_2^*, \quad  \T_3 = U_3 \Sigma_2 W_3^*$.

The HOSVD procedure for the construction of a core ${\mathcal C}$ of $\T$ consists then of the following multilinear multiplication: $$(U^*_1,U^*_2,U^*_3) \cdot \T = \mathcal{C}.$$
\subsection{Multiview Geometry and Grassmann Tensors}
\label{prelimCV}

For the convenience of the reader we recall standard facts and notation in the context of projective reconstruction in computer vision. A
{\it scene} is a set of $N$ points $\{\nbXi \}_{i=1,
\dots, N}$ in $\Pin{k}=\mathbb{P}(W),$ where $W$ is a vector space of dimension $k+1.$ A {\it camera} $P$ is a projection from $\Pin{k}$ onto the target space ({\it view})
$\Pin{h} = \mathbb{P}(\mathcal{V}),$ where $\mathcal{V}$ is a vector space of dimension $h +1, h < k,$ from a linear center $C_P = \mathbb{P}(K),$ where $K$ is a vector space of dimension $k-h.$ Once bases have been chosen in $W$ and $\mathcal{V},$ $P$ can be identified with
a $(h+1) \times (k+1)-$ matrix of maximal rank, defined up to a
constant, for which we use the same symbol $P$. With this
notation, $K$ is the right annihilator of $P,$ and using the same notation $\mathbf{X}$ for the point's homogeneous coordinates in the chosen bases, $P(\mathbf{X})$ denotes the image $P\cdot \mathbf{X}$ of a point $\mathbf{X}$ in
$\Pin{k}.$ 

In the context of multiple view geometry, one considers a set of
multiple images of the same scene, obtained from a set of cameras
$P_j:\Pin{k}\setminus C_j \to \Pin{h_j}$ where $\Pin{k} = \mathbb{P}(W),$ $\Pin{h_j} = \mathbb{P}(\mathcal{V}_j)$, and  $C_j = \mathbb{P}(K_j).$ Two different images 
$P_l(\mathbf{X})$ and $P_m(\mathbf{X})$ of
the same point $\mathbf{X}$ are \textit{corresponding points} and,
more generally, $r$ linear subspaces $\mathcal{S}_j \subset
\Pin{h_j},$ $j=1,\dots, r$ are said to be \textit{corresponding}
if there exists at least one point $\mathbf{X} \in \Pin{k}$ such
that $P_j(\mathbf{X})\in \mathcal{S}_j$ for $j=1,\dots, r.$
In \cite{Hart-Schaf} Hartley and Schaffalitzky
introduced {\it Grassmann tensors} which
encode the relations between sets of corresponding subspaces in
the various views. We recall here the basic elements of their
construction.

Consider, as above, a set of projections $P_j:\Pin{k}\setminus{C_j} \to
\Pin{h_j},$ $j = 1,\dots,r,$ $h_j \geq 2$ and a {\it profile},
i.e. a partition $(\alpha_1, \alpha_2, \dots, \alpha_r)$ of $k+1,$
where $1 \leq \alpha_j \leq h_j$ for all $j,$ and $\sum\alpha_j =
k+1.$

Let $\{\mathcal{S}_j\},$ $j=1,\dots,r,$ where $\mathcal{S}_j
\subset \Pin{h_j},$ be a set of general $s_j$-spaces, with
$s_j=h_j-\alpha_j,$ and let $S_j$ be a maximal rank
$(h_j+1)\times (s_j+1)-$matrix whose columns are a basis for
$\mathcal{S}_j$. By definition, if all the $\mathcal{S}_j$ are
corresponding subspaces there exist a point $\mathbf{X} \in
\Pin{k}$ such that $P_j(\mathbf{X})\in \mathcal{S}_j$ for
$j=1,\dots, r.$ In other words there exist $r$ vectors
$\mathbf{v_j} \in \mathbb{C}^{s_j+1}$ $j = 1,\dots,r,$ such that:
\begin{equation}
\label{grasssystem}
\left[%
\begin{array}{ccccc}
  P_1 & S_1& 0 & \dots & 0\\
  P_2 & 0 & S_2 & \dots & 0 \\
  \vdots & \vdots & \vdots & \vdots & \vdots \\
  P_r & 0 & \dots & 0 & S_r
\end{array}
\right]
			\cdot
\left [
\begin{array}{c}
				\mathbf{X}\\
				\mathbf{v_1} \\
				\mathbf{v_2} \\
				\vdots \\
				\mathbf{v_r} \\
\end{array}
\right]
			=
   \left[
\begin{array}{c}
				0 \\
				0 \\
				\vdots \\
				0 \\
\end{array}
\right]
		\end{equation}

The existence of a non trivial solution
$\{\mathbf{X},\mathbf{v_1},\dots,\mathbf{v_r}\}$ for  system
(\ref{grasssystem}) implies that the system matrix has zero
determinant. This determinant can be thought of as an $r$-linear
form, i.e. a tensor, in the Pl\"{u}cker coordinates of the spaces
$\mathcal{S}_j.$ This tensor is called the {\it Grassmann tensor} $\T$ with profile $(\alpha_1, \dots, \alpha_r).$ 

More explicitly, the entries of the Grassmann tensor are computed as maximal minors of the matrix:
\begin{equation}
\left[%
\begin{array}{c|c|c|c}
\label{matricetrasposta_r}
  {P_1}^T & {P_2}^T & \dots & {P_r}^T \\
\end{array}%
\right],
\end{equation}
 obtained by selecting $\alpha_j$ columns
from ${P_j}^T$, for $j=1, \dots, r.$ 
Notice that each column in $P_j^T$ can be thought of as an element in ${\mathbb P}(V_j)$, where $V_j= \bigwedge^{s_j+1}((W/K_j)^*)$ is the vector space of dimension
$n_j =\binom{h_j+1}{h_j-\alpha_j + 1}=\binom{h_j+1}{s_j + 1}$ such that the Grassmannian $G(s_j, h_j) \subset \mathbb{P}(V_j)=\mathbb{P} (\bigwedge^{s_j+1}((W/K_j)^*).$ As a consequence,
$\T \in V_1 \tensor V_2\tensor ... \tensor V_r$. Therefore for each $j=1, \ldots, r,$ an $s_j$-dimensional subspace ${\mathcal S}_j$ can be described as the intersection of $\alpha_j=h_j-s_j$ hyperplanes of $\Pin{k}$ containing $\mathbb{P}(K_j).$  In other words, the columns of each $P^T_j$ may be viewed as hyperplanes of ${\mathbb P}^k$ containing the center $C_j$. Moreover, the choice of $\alpha_j$ columns of $P^T_j$ gives an element in $Gr(\alpha_j-1, h_j) \subset {\mathbb P}(\bigwedge^{\alpha_j}(W/K_j))$ which is the dual Grassmannian of $Gr(s_j,h_j).$ 

It is useful to observe that a right action of
$GL(k+1)$ on \brref{matricetrasposta_r}, i.e.  a change of
coordinates in the ambient space $\Pin{k},$  does not alter the tensor,
as all entries are multiplied by the same nonzero constant.

As far as the effect of changes of coordinates in each of the view we have the following remark:
\begin{remark}
\label{invfrk2}
   The F-$\rk\T$ is invariant under change of coordinates in each of the views $\Pin{h_r}.$ From Remark \ref{invfrk} it is enough to show that any left action of $GL(h_j + 1)$ on  $P_j^{T},$ i.e. a change of coordinates in the corresponding view, induces a linear invertible transformation on $V_j.$ Indeed, any transformation $H_j \in GL((W/K_j))$ yields the transformation $\bigwedge^{\alpha_j}H_j$ on the Plu\"cker coordinates of $Gr(\alpha_j-1,h_j)$. Since the tensor is expressed in terms of the Pl\"ucker coordinates of the Grassmanniann $Gr(s_j,h_j)$, the transformation induced on the tensor by the matrix $H_j$ is $\bigwedge^{s_j+1} \left( H_j^*\right)^{-1}$, where the adjoint (transpose) is needed because of the dual coordinates and where the inverse appears because of the action on the coefficients of the tensor ${\mathcal T}$.
\end{remark}

\subsection{Canonical Form of Trifocal Grassmann Tensors}
\label{canformsec}
In \cite{BBBT1} the authors showed that, under some generality assumptions, one can obtain a canonical form $\T^c$ for a trifocal Grassmann tensor $\T$ that leads to a direct computation of its rank. It turns out that the same canonical form allows us to successfully compute the multilinear rank of $\T$ as well, under the same genericity assumption. For the convenience of the reader here we summarize the construction of the canonical form $\T^c$ for $\T,$ referring the reader to \cite{BBBT1} for additional details.

Let $\T$ be a trifocal Grassman tensor corresponding to projection matrices $P_j:\Pin{k}\setminus C_j \to \Pin{h_j},$ $j=1,2,3,$ with profile $(\alpha_1,\alpha_2, \alpha_3),$ and let $L_1$,
$L_2$ and $L_3$ be the vector spaces of dimension $h_1+1$, $h_2+1$
and $h_3+1$ respectively, spanned by the columns of ${P_1}^T$,
${P_2}^T$ and ${P_3}^T.$ Let $\Lambda_1=\mathbb{P}(L_1)$,
$\Lambda_2=\mathbb{P}(L_2)$ and $\Lambda_3=\mathbb{P}(L_3)$.

We consider, for each triplet of distinct integers $r,s,t\in\{1,2,3\}$
the following integers:

\begin{align}
&i_{r,s}=h_r+h_s+1-k;\label{numrel1}\\
&i=h_1+h_2+h_3+1-2k;\label{numrel2} \\
&j_{r,s}=i_{r,s}-i=k-h_t.\label{numrel3}
\end{align}
Notice that the definition of $j_{r,s}$ is independent of the order of the indices, i.e. $j_{r,s} = j_{s,r}.$
\noindent Our generality assumption is the following:
\begin{assumption}

 \label{g.a.} For any choice of $r,s,t$ with $\{r,s,t\}=\{1,2,3\}$, $L_t$ and the intersection $\Lambda_{rs}=L_r \cap L_s$  span $\mathbb{C}^{k+1},$ or, equivalently, the linear span of each
pair of centers does not intersect the third one.
\end{assumption}

\noindent This assumption implies, in particular, that for any
choice of a pair $r,s$, the span of $L_r$ and $L_s$ is the whole
$\mathbb{C}^{k+1},$ or, in other words, that the two centers $C_r$
and $C_s$ do not intersect.

Under Assumption \ref{g.a.}, applying Grassmann formula one sees that the three
numbers above have the following meaning: $i_{r,s}= dim (L_r \cap
L_s) \geq 0$, for any choice of $r,s$ , $i= dim (L_1 \cap L_2 \cap
L_3) \geq 0$ and $j_{r,s}$ is the affine dimension of the center
$C_t$ i.e. $k-{h_t}=j_{rs}$ for $r,s,t=1,2,3.$

In \cite{BBBT1} it is shown that under Assumption \ref{g.a.} a suitable choice of bases, realized by $H_j \in GL(h_j+1),$  for $j=1,2,3,$ and $K \in GL(k+1),$ transforms the matrix \brref{matricetrasposta_r} as
\begin{equation}
\Phi^k_{h_1,h_2,h_3}:=\left[(H_1P_1K)^T|(H_2P_2K)^T|(H_3P_3K)^T\right].
\end{equation}
so that
\begin{equation}
\label{canmat}
\Phi^k_{h_1,h_2,h_3}:=\left[%
\begin{array}{ccc|ccc|ccc}
  I_i & \mathbf{0} & \mathbf{0} & I_i & \mathbf{0} & \mathbf{0} & I_i & \mathbf{0} & \mathbf{0} \\
  \mathbf{0} & I_{j_{1,2}} & \mathbf{0} & \mathbf{0} & I_{j_{1,2}} & \mathbf{0} & \mathbf{0} & \mathbf{0} & \mathbf{0}\\
  \mathbf{0} & \mathbf{0} & I_{j_{1,3}} & \mathbf{0} & \mathbf{0} & \mathbf{0} & \mathbf{0} & I_{j_{1,3}} & \mathbf{0}\\
  \mathbf{0} & \mathbf{0} & \mathbf{0} & \mathbf{0} & \mathbf{0}& I_{j_{2,3}} & \mathbf{0} & \mathbf{0} & I_{j_{2,3}}\\
\end{array}%
\right].
\end{equation}

As described above, entries of $\T^c$ are given by the maximal minors of \brref{canmat}, obtained by selecting $\alpha_j$ columns
from ${H_jP_jK}^T$, for $j=1,2, 3.$ More precisely, as in \cite{BBBT1}, let $(a_1,a_2,a_3)$ be a partition of $\alpha_1$ and let $(b_1,b_2,b_3)$ and $(c_1,c_2,c_3)$ be partitions of $\alpha_2$ and $\alpha_3,$ respectively. Each entry of  $\T^c$ is a maximal minor $\T^c_{I,J,K}$ of \brref{canmat} built by choosing $a_1$ columns from $I_i,$ $a_2$ columns from $I_{j_{1,2}},$ $a_3$ columns from $I_{j_{13}},$ appropriately completing them with zero vectors to obtain entire columns of \brref{canmat} and proceeding analogously with $b_1,b_2,b_3$ and the second block of \brref{canmat} and with $c_1,c_2,c_3$ and the third block of \brref{canmat}, where  $I=(i_1, \dots, i_{s_1+1}),$ $J=(j_1, \dots,
j_{s_2+1}),$ $K=(k_1, \dots, k_{s_3+1}),$ with $1 \leq i_1 < \dots < i_{s_1+1} \leq
h_1+1$, $1 \leq j_1 < \dots < j_{s_2+1} \leq h_2+1$ and $1 \leq
k_1 < \dots < k_{s_3+1} \leq h_3+1$ are the indices of the columns of the three blocks of \brref{canmat} that were {\it not} chosen. 
As already recalled in \cite{BBBT1}, the entries $\T^c_{I,J,K}$ of the tensor ${\mathcal T^c}$ are indexed with respect to the lexicographical order of the families of multi-indices $\{I\}, \{J\},$ and $\{K\}$.
If we consider the first flattening $\T^c_1$ of $\T^c,$ one then sees that a row of $\T^c_1$ corresponds to one specific choice of  $a_1$ columns from $I_i,$  $a_2$ columns from $I_{j_{12}},$ and $a_3$ columns from $I_{j_{1,3}},$  with $a_1+a_2+a_3=\alpha_1$. Similarly one sees the role of specific choices of $b_u$ columns and $c_u$ columns from the corresponding submatrices of $\Phi^k_{h_1,h_2,h_3}$ in determining a chosen row of $\T^c_2$ and $\T^c_3$ respectively.

\begin{remark}
In the following section we will make use of the canonical form (\ref{canmat}) in order to determine the multilinear rank of a Grassmann tensor satisfying Assumption \ref{g.a.}. As mentioned in \cite{BBBT1}, if Assumption \ref{g.a.} doesn't hold, we cannot obtain a canonical form depending only on the dimension of the various spaces and, indeed, even the rank of the Grassmann tensor depends also on the geometric configuration of the three projections. 

This observation is still true as far as the multilinear rank is concerned; this is the reason why in this paper we will always assume that Assumption \ref{g.a.} is satisfied.

As an example, consider the case of three projections from $\mathbb{P}^4$ to  $\mathbb{P}^2$, with profile $(2,2,1).$ Notice that in this case $i=-1$. We can choose projection matrices $P_j$, $j=1,2,3$ such that:

$$	\label{non_gen}
	P_1^T:=\left[%
	\begin{array}{ccc}
		1 & 0 & 0 \\
		0 & 1 & 0 \\
		0 & 0 & a \\
		0 & 0 & b \\
		0 & 0 & c \\
			\end{array}%
	\right], 
		P_2^T:=\left[%
	\begin{array}{ccc}
		1 & 0 & 0 \\
		0 & 0 & d \\
		0 & 1 & 0 \\
		0 & 0 & e \\
		0 & 0 & f \\
	\end{array}%
	\right], 
		P_3^T:=\left[%
	\begin{array}{ccc}
		0 & 0 & g \\
		1 & 0 & 0 \\
		0 & 1 & 0 \\
		0 & 0 & h \\
		0 & 0 & k \\
	\end{array}%
	\right],$$ 

\noindent with $(a,b,c) \neq (0,0,0), (d,e,f) \neq (0,0,0), (g,h,k) \neq (0,0,0).$
\end{remark}
The first flattening of the corresponding trifocal tensor is

$$\left[%
\begin{array}{ccccccccc}
	g(ce-bf) & a(fh-ek) & ch-bk & 0 & bf-ce & 0 & 0 & 0 & 0 \\
d(bk-ch) & 0 & 0 & 0 & 0 & ce -bf& 0 & 0 & 0 \\
ek-fh & 0 & 0 & 0 & 0 & 0 & 0 & 0 & 0 \\
\end{array}%
\right],$$

\noindent whose rank is generically  $3$ and drops to at most $2$ if $ek=fh.$

\section{The multilinear rank of trifocal Grassmann tensors}\label{multilin}
In this section, $\T$ will always be a trifocal Grassmann tensor of dimension $n_1\times n_2\times n_3,$ with profile $(\alpha_1,\alpha_2,\alpha_3),$ satisfying Assumption \ref{g.a.} and $\T^c$ will be its canonical form introduced in Section \ref{canformsec}. Leveraging properties of $\T^c$ we will obtain results on F-$\rk{\T.}$
\label{multilinrk}
\begin{Lemma}
\label{lemmanminusv}
Let $\T^c$ be the canonical from of a trifocal Grassmann tensor $\T$ of dimension $n_1\times n_2\times n_3,$ satisfying Assumption \ref{g.a.}. Then the multilinear rank is F-$\rk{\T^c} = (n_1 - v_1, n_2 - v_2, n_3 - v_3)$ where $v_r$ is the number of zero rows in the flattening matrix $\T^c_r.$
\end{Lemma}
\begin{proof}
In the proof of \cite[Theorem 5.2]{BBBT1} the authors showed that, with our assumptions on $\T,$
if $\T^c_{\hat{i},\hat{j},\hat{k}} \neq 0$ then $\T^c_{\hat{i},\hat{j}, k} =0$ for all $k \neq \hat{k}.$ Considering the cyclic role of the indices $i,j,k$ the above observation also says that $\T^c_{i,\hat{j}, \hat{k} }=0$ for all $i \neq \hat{i},$ and $\T^c_{\hat{i},j, \hat{k}} =0$ for all $j \neq \hat{j}.$
 Assume $T^c_{\hat{i},\hat{j},\hat{k}} \neq 0,$ then the above observation can be visualized in $\T^c_1$ as follows:

$$
\tiny{
\left[
\begin{array}{cccc|cccc|cccc|cccc}
..&.. &.. & ..&..&..&.. &.. & ..&0&..&.. &.. & ..&..&..\\
..&0_{\hat{i},\hat{j},1}&..&.. &.. &0_{\hat{i},\hat{j},2}&..&..&0 &*_{\hat{i},\hat{j},\hat{k}}&0&0&.. &0_{\hat{i},\hat{j},{n_3}} & ..&..\\
..&.. &.. & ..&..&..&.. &.. & ..&0&..&.. &.. & ..&..&..\\
..&.. &.. & ..&..&..&.. &.. & ..&0&..&.. &.. & ..&..&..\\
\end{array}%
\right]}
$$
 While $\T^c_1$ can have more than one nonzero element on the same row, it cannot contain two nonzero elements on the same column. Hence any two rows containing non-zero elements are linearly independent. Therefore, if $v_1$ is the number of zero rows of $\T^c_1,$ it is  $\rk{\T^c_1} = n_1 - v_1.$ A similar argument can be carried out for $\T^c_2$ and $\T^c_3.$
 \end{proof}

 \begin{Lemma}
 \label{lemmajrsconditions}
Let $\T^c$ be the canonical form of a trifocal Grassmann tensor $\T$ of dimension $n_1\times n_2\times n_3,$ with profile $(\alpha_1,\alpha_2,\alpha_3),$ satisfying Assumption \ref{g.a.}. Let $(r,s,t)$ be any permutation of $\{1,2,3\}$ and let $j_{r,s}$ be defined as in \brref{numrel3}. Then the flattening matrix $\T^c_r$ contains zero rows if and only if

\begin{alignat}{1}
\label{jrineq} 
j_{r,s} - \alpha_s - 1 &\geq \max (0,\alpha_r -i-j_{r,t} ) \ \ \ \text{      or      }\\
 j_{r,t} - \alpha_t + 1 &\geq \max (0,\alpha_r -i-j_{r,s} ).\nonumber
\end{alignat}
 Moreover, conditions \brref{jrineq} are mutually exclusive.
 \end{Lemma}
 \begin{proof}
For simplicity, let us fix $(r,s,t) = (1,2,3)$  and conduct the proof for $\T^c_1$, noticing that the proof is identical for $\T^c_2$ and $\T^c_3,$ with a cyclic adjustment of the role of the three indices and of the parameters $\{a_j,b_j,c_j\}$ introduced in Section \ref{canformsec}.
Recall that $\T^c_1$ is the first flattening matrix of $\T^c = [\T^c_{\ell,j,k}],$ of dimension $n_1 \times (n_2  n_3)$ obtained by juxtaposing $n_3$ blocks of dimension $n_1 \times n_2,$  where $\ell$ runs over the rows of $\T^c_1,$ $j$ runs over the columns of each block, and $k$ runs over the blocks. Let $i$ be as defined in \brref{numrel2}, and let $\{a_j,b_j,c_j\}, $ $j=1,2,3$ be as in Section \ref{canformsec}.
As noted in Section \ref{canformsec}, a row of $\T^c_1$ corresponds to one specific choice of  $a_1$ columns from the first block of \brref{canmat},  $a_2$ columns from its second block, and $a_3$ columns from its thrid block,  with $a_i\geq 0,$ and $a_1+a_2+a_3=\alpha_1$. \\
Suppose $j_{13} - \alpha_3 -1 \geq \max (0,\alpha_1 -i-j_{1,2} )$, so that $j_{13} \geq \alpha_3 + 1,$ and let $a_1,a_2,a_3$ be such that $0\leq a_3 \leq j_{1,3} - \alpha_3 - 1.$ Notice that the assumption  $j_{13} - \alpha_3 -1 \geq \max (0,\alpha_1 -i-j_{1,2} )$ implies that there exists at least one such triplet with $a_1 + a_2 + a_3 = \alpha_1.$ Because $c_2 \leq \alpha_3,$ it is
\begin{equation}
\label{a3c2}
a_3 + c_2 \leq j_{1,3} -1.
\end{equation}
Recalling the canonical structure of the matrix \brref{canmat}, it follows from \brref{a3c2} that all elements of the row of $\T^c_1$ corresponding to a choice of $a_1,a_2,a_3$ as above are zero, as all the maximal minors corresponding to elements of this row are now forced to contain at least one duplicate column coming from $I_i, I_{j_{1,2}}$ or $I_{j_{2,3}}.$
The proof can be carried out with the obvious adjustments if $j_{1,2} -\alpha_2 - 1 \ge \max(0, \alpha_1 -i-j_{1,3}).$

Assume now that
\begin{alignat}{1}
j_{1,2} -\alpha_2 - 1 &< \max (0, \alpha_1 -i-j_{1,3}) \ \ \ \text{      and     } \label{max1}\\
j_{1,3} - \alpha_3 -1 &< \max(0, \alpha_1 -i-j_{1,2})\label{max2}.
\end{alignat}
We will show that every row in $\T^c_1$ contains at least one non-zero element. Let us fix a row of $\T^c_1$ by fixing non-negative values for $(a_1,a_2,a_3)$ with $\sum_\ell a_\ell = \alpha_1,$  $a_1 \le i, a_2 \leq j_{1,2},$ and $a_3 \leq j_{1,3}.$ As observed in \cite{BBBT1}, this row contains a non  zero element if the following system of linear equations
\begin{equation}
\label{systemabc}
\left\{
\begin{array}{l}
b_1+c_1= i -a_1\\
b_2= j_{1,2}- a_2\\
c_2  = j_{1,3}-a_3\\
b_3+c_3  = j_{2,3}\\
b_1+b_2+b_3  = \alpha_2\\
c_1+c_2+c_3  = \alpha_3\
\end{array}
\right .
\end{equation}
has at least one set of integer solutions in the unknowns $(b_1, b_2, b_3, c_1, c_2, c_3),$ that satisfy the following conditions:
\begin{align}
\label{condonbc}
 &0\leq  b_1\leq i       &       0 \leq   b_2\leq j_{1,2}&      &      0 \leq   b_3 \leq j_{2,3} \\
 &0\leq  c_1\leq i       &       0 \leq   c_2\leq j_{1,3}&      &       0 \leq   c_3 \leq j_{2,3}.\notag 
\end{align}
Our assumptions on $(a_1,a_2,a_3)$ imply that the second and third equation of \brref{systemabc} are already solved, satisfying the relevant \brref{condonbc}. Therefore it remains to show that it is possible to choose $0\leq c_3 \leq j_{2,3}$ such that 
$b_1= \alpha_2 - j_{1,2} + a_2 - j_{2,3} + c_3$ and $c_1 = \alpha_3 - j_{1,3} + a_3 - c_3$
satisfy the relevant \brref{condonbc}, i.e. 
\begin{equation}
\label{condonb1c1}
0\le \alpha_2 - j_{1,2} + a_2 - j_{2,3} + c_3 \le i\ \ \ \text{and} \ \ \ 0\le \alpha_3 - j_{1,3} + a_3 - c_3 \le i.
\end{equation}

Recalling that $i = \sum_\ell \alpha_{\ell} - j_{1,2} - j_{1,3} -j_{2,3}$, \brref{condonb1c1} give:
\begin{align}
j_{1,2} +j_{2,3} -\alpha_1 -\alpha_2 + a_3 &\leq c_3 \leq \alpha_1 + \alpha_3 - j_{1,3} - a_2 \label{B1}\\
j_{1,2} +j_{2,3} -\alpha_2 -a_2 &\leq c_3 \leq \alpha_3 - j_{1,3} +a_3. \label{B2}
\end{align}
As $\sum_\ell a_{\ell} = \alpha_1$ and $a_1 \ge 0,$ it follows that \brref{B2} imply \brref{B1}, hence, setting $LB=j_{1,2} +j_{2,3} -\alpha_2 -a_2$ and $UB=\leq \alpha_3 - j_{1,3} +a_3,$ we need to show that under assumptions \brref{max1} and \brref{max2} we can choose $c_3$ such that 
\begin{equation}
\label{condonc3}
0\le c_3\le j_{2,3}  \text{ and } LB\leq c_3 \leq UB.
\end{equation}

First notice that if $\alpha_1 - j_{1.2} - i >0$ then $j_{2,3} \ge \alpha_2$ and $LB \geq 0.$ Indeed in this case it is $\max (0, \alpha_1 - j_{1,2} -1) = \alpha_1 - j_{1,2} -i$ and \brref{max1} gives $j_{2,3} \geq \alpha_2.$ Because $j_{12} - a_2 \geq 0$ then it is $LB \geq 0.$ A very similar argument, using \brref{max2}, shows that if $\alpha_1 - j_{1.3} - i >0$ then $j_{2,3} \ge \alpha_3$ and $UB \leq j_{2,3}.$\\
Four different cases, according to the respective sign of $\alpha_1 - j_{1,2} - i$ and $\alpha_1 - j_{1,3} - i,$ need to be considered, as in the table below:
\begin{center}
\begin{tabular}{|c|c|c|}
\hline
Case& $\alpha_1 - j_{1,2} - i $ & $\alpha_1 - j_{1.3} - i$\\
\hline
1&$>0$&$>0$\\
2&$>0$&$\leq 0$\\
3&$\leq 0$&$>0$\\
4&$\leq 0$&$\leq 0$\\
\hline
\end{tabular} 
\end{center}

{\it Case 1.} From above it is $LB\geq0 $ and $UB\leq j_{2,3},$ hence one can choose any value $LB\leq c_3 \leq UB$ to satisfy \brref{condonc3}.

{\it Case 2.} From above it is $LB \geq 0.$ Choose $c_3 = LB.$ Because $\alpha_1 - j_{1,3} - i \leq0$ it is $\max(0,\alpha_1 - j_{1,3} - i) =0$ and thus \brref{max2} gives $j_{1,2} - \alpha_2 \leq 0.$ As $a_2 \geq 0,$ it is $LB \leq j_{2,3}$ and \brref{condonc3} are satisfied.

{\it Case 3.} From above it is $UB \leq j_{2,3}.$ Choose $c_3 = UB.$ Because $\alpha_1 - j_{1,2} - i \leq0$ it is $\max(0,\alpha_1 - j_{1,2} - i) =0$ and thus \brref{max1} gives $j_{1,3} - \alpha_3 -1 \leq 0.$ As $a_3 \geq 0,$ it is $UB \geq 0$ and \brref{condonc3} are satisfied.

{\it Case 4.} 
In this case we need to further consider the possible relative sign of $LB$ and $UB-j_{2,3},$ generating four possible cases as in the table below. In each case one can choose $c_3$ as indicated in the fourth column. Arguments similar to the ones used above in previous cases show that \brref{condonc3} are satisfied.
\begin{center}
    \begin{tabular}{|c|c|c|c|}
    \hline
    Case&LB&UB&$c_3$\\ \hline
    $i$     &$\geq 0$   &$\leq j_{2,3}$ & any $LB\leq c_3\leq UB$ \\
    $ii$    &$\geq 0$   &$> j_{2,3}$    & $c_3 = LB$ \\
    $iii$   &$< 0$      &$\leq j_{2,3}$ & $c_3 = UB$ \\
    $iv$    &$< 0$      &$> j_{2,3}$    & any $0\leq c_3\leq j_{2,3}$  \\
    \hline 
    \end{tabular}
\end{center}
Finally, notice that conditions \brref{jrineq} cannot both hold. If they did, then for any $0\leq a_2 \le j_{1,2}-\alpha_2 -1$ and $0\leq a_3 \leq j_{1,3} - \alpha_3 -1$ it would follow that:
\begin{equation}
\label{a2a3}
 a_2+a_3 \leq j_{1,2} + j_{1,3} - \alpha_2 - \alpha_3 -2.
  \end{equation}
  Recalling that $a_1 + a_2 + a_3 = \alpha_1,$  $\sum_j \alpha_j = k+1= i + j_{1,2} + j_{1,3} + j_{2,3},$ \brref{a2a3} would give $i+j_{2,3} \le a_1 -2,$ which is impossible as $a_1 \leq i.$
 \end{proof}
 \begin{claim}
      \label{remsetA}
 With the notation of this section, let $i$ and $j_{u,v}$ for $u,v\in\{1,2,3\},$ $u\neq v,$  be defined as in \brref{numrel2} and\brref{numrel3} and let $\T^c_1$ be the first flattening matrix of $\T^c.$
  Assume that $\rk{\T^c_1}<n_1,$ i.e. $\rk{\T^c_1}$ is not maximum, so that, by Lemma \ref{lemmajrsconditions}, $ j_{1,s} - \alpha_s -1 \geq \max (0, \alpha_1 -i -j_{1,t})$ for some $s,t \in\{2,3\}$ $t \neq s.$ 
Let
\begin{equation*}
\begin{split}
A= \{ (a_1,a_2, a_3) |&\, a_u\in \mathbb{Z}_{\geq 0},\sum_u a_u = \alpha_1, a_1 \le i, \\
&\max(0,\alpha_1 -i-j_{1,t})  \leq a_s\leq j_{1,s} -\alpha_s-1, \\ 
&a_t\le j_{1,t} \}
\end{split}
\end{equation*}
The cardinality $|A|$ of the set $A$ can be computed as follows. For each $(a_1,a_2,a_3) \in A$ set $m_1 = \min{(i,\alpha_1-a_s)}$ and $m_2 = \min{(j_{1,t}, \alpha_1-a_s)}.$
 Then
 \begin{equation*}
  |A|= \sum_{\max(0,\alpha_1 -i-j_{1,t})\le a_s \le j_{1,s} -\alpha_s -1} N_s
  \end{equation*}
  where
  \begin{equation*}
  N_s=
 \begin{cases}
 \alpha_1 - a_s + 1 & \text{ if } m_1 = m_2 = \alpha_1 - a_s \\
 i+1 & \text{ if }  m_1 = i \text { and } m_2 = \alpha_1 - a_s\\
 j_{1s} + 1 & \text{ if }  m_1 = \alpha_1 - a_s \text{ and } m_2 = j_{1,t}\\
 |i-\alpha_1+a_s+j_{1,t}|+1 & \text{ if } m_1 = i \text{ and } m_2 = j_{1,t}
 \end{cases}
 \end{equation*}
Similarly, one can define sets $B$ and $C,$ respectively, for flattenings $\T^c_2$ and $\T^c_3,$  if their ranks are not maximum. In those cases $b_u$ and $c_u$ play the role of $a_u$ and the $j_{u,v} $ are adjusted accordingly, taking into consideration their role in \brref{canmat} and in Lemma \ref{lemmajrsconditions}.
 \end{claim}

\begin{theorem}
\label{Frank}
Let $\T$ be a trifocal Grassmann tensor of dimension $n_1\times n_2\times n_3,$ with profile $(\alpha_1,\alpha_2,\alpha_3),$ satisfying Assumption \ref{g.a.}. Let  $i$ and $j_{u,v}$ for $u,v\in\{1,2,3\},$ $u\neq v,$ be defined as in \brref{numrel2} and \brref{numrel3}. Then the rank of the first flattening $\T_1$ is:
\begin{equation}
\rk{\T_1}= n_1 - \sum_{A} \bincof{i}{a_1} \bincof{j_{1,2}}{a_2} \bincof{j_{1,3}}{a_3}
\end{equation}
where $A= \{ (a_1,a_2, a_3) |\, a_u\in \mathbb{Z}_{\geq 0}, \sum_u a_u = \alpha_1, a_1 \le i, \max(0,\alpha_1 -i-j_{1,t})  \leq a_s\leq j_{1,s} -\alpha_s-1, a_t\le j_{1,t} \}.$
\end{theorem}
\begin{proof}
Let $\T^c$ be the canonical form of $\T$ and $\T^c_1$ be its first flattening matrix. Remark \ref{invfrk2} shows that  F-$\rk{\T^c_1} =$ F-$\rk{\T_1},$ therefore from now on we will focus on F-$\rk{\T^c_1}.$ 
From Lemma \ref{lemmanminusv} the rank of the flattening matrix $\T_1^c$ is known if the number $v_1$ of zero rows of $\T_1^c$ is known.  Assume $j_{1,s}- \alpha_s - 1 \geq \max(0,\alpha_1 -i-j_{1,t}) $ and let $A$ be the corresponding set defined in Claim \ref{remsetA}, which, under our last assumption, is non empty. As noted above, choosing a row of $\T^c_1$ is equivalent to choosing $a_1$ columns from $I_i,$ $a_2$ columns from $I_{j_{1,2}},$ $a_3$ columns from $I_{j_{13}},$ appropriately completing them with zero vectors to obtain entire columns of \brref{canmat},  where $a_i\geq 0,$ $a_1+a_2+a_3=\alpha_1,$ $a_1\leq i,$ $a_2 \leq j_{1,2},$ and $a_3 \leq j_{1,3}.$
From Lemma \ref{lemmajrsconditions} and Claim \ref{remsetA} it follows that zero rows in $\T^c_1$ are exactly all rows that correspond to triplets of nonnegative integers $(a_1,a_2,a_3 )\in A,$ hence $$\rk{\T^c_1} = n_1 - \sum_{A} \bincof{i}{a_1} \bincof{j_{1,t}}{a_t} \bincof{j_{1,s}}{a_s}.$$ 

If $j_{1,2}- \alpha_2 - 1 < \max(0,\alpha_1 -i-j_{1,3})$ and $j_{1,3}- \alpha_3 - 1 < \max(0,\alpha_1 -i-j_{1,2})$ then $A$ is the empty set and thus $v_1 = 0,$ and F-$\rk{\T^c_1}$ is maximum, i.e. F-$\rk{\T_1} = n_1.$
\end{proof}
\begin{remark}
    While Theorem \ref{Frank} gives the result for the first flattening, one can easily obtain the rank of the second and third flattening matrices by simply switching the order of the views and proceeding accordingly.
\end{remark}
\begin{Cor}
    Let $\T$ be a trifocal Grassmann tensor of dimension $n_1\times n_2\times n_3,$ with profile $(\alpha_1,\alpha_2,\alpha_3),$ satisfying Assumption \ref{g.a.}. Then for at least one $r \in \{1,2,3\}$ it is $\rk{\T_r} =n_r.$
\end{Cor}
\begin{proof}
    Suppose that $\rk{\T_r} < n_r$ for all $r = 1,2,3.$ From Lemma \ref{lemmajrsconditions} it follows that three of the six conditions 
    \begin{equation}
        (r,s,t) :j_{r,s} - \alpha_s -1 \geq \max(0,\alpha_r - i - j_{r,t}) ,
    \end{equation}
    where $r,s,t \in \{1,2,3\},$ must hold, one with $r=1,$ one with $r=2,$  and one with $r =3.$
    First observe that, for fixed values of $r,s,t,$ if $(r,s,t)$ holds then neither $(s,t,r)$ nor $(t,r,s)$ can hold. Indeed assume $(r,s,t)$ holds and $\alpha_r - i - j_{r,t} = j_{r,s} - \alpha_s + j_{s,t} - \alpha_t <0.$ Then $\max(0,\alpha_r - i - j_{r,t}) =0,$ and $j_{r,s} -\alpha_s -1 \geq 0,$ which in turn gives $j_{s,t} -\alpha_t <0,$ and thus $(s,t,r)$ can not hold. Assume instead that $(r,s,t)$ holds and $\alpha_r - i - j_{r,t} = j_{r,s} - \alpha_s + j_{s,t} - \alpha_t \geq 0.$ Then $(r,s,t)$ gives $j_{s,t}-\alpha_t \leq -1,$ and thus $(s,t,r)$ is not possible in this case either. 
    Further observe that $(r,s,t)$ implies $j_{r,s} - \alpha_s \geq 1$ and if $(t,r,s)$ held it would be $j_{t,r} - \alpha_r - 1 \geq j_{t,r} - \alpha_r + j_{r,s} - \alpha_s$ and thus $j_{r,s} -\alpha_s \leq -1$ which is impossible. Hence if $(r,s,t)$ holds neither $(s,t,r)$ nor $(t,r,s)$ can hold.
    Now assume $(r,s,t)$ is one of the six conditions that hold. From the above observation it follows that $(s,r,t)$ must hold. But the same observation then implies that $(t,s,r)$ must hold, which is incompatible with $(s,r,t),$ again from the above observation.
    \end{proof}

\begin{remark}
Let $\T$ be a trifocal Grassmann tensor of dimension $n_1\times n_2\times n_3,$ with profile $(\alpha_1,\alpha_2,\alpha_3),$ satisfying Assumption \ref{g.a.}. As its canonical form $\T^c$ is obtained via successive invertible transformations in the ambient space and in the views, it is $\rk{\T} = \rk{\T^c}$.
\end{remark}

\begin{remark} 
\label{remfindvanishingrows} Proposition \ref{Frank} and Claim \ref{remsetA} show how to count the number of zero rows in a flattening matrix $\T^c_r$ of a tensor $\T^c$ in canonical form. Here we describe a procedure that identifies exactly which rows of the flattening matrix vanish. For simplicity we will set $r=1,$ as similar arguments work for $r=2,3.$ Let ${\mathcal T}^c_1$ be a flattening matrix of a tensor $\T^c$ as above, and let $\underline{a}=(a_1, a_2, a_3) \in A$, where $A$ is as in Claim \ref{remsetA}.  Recall that the rows of ${\mathcal T}^c_1$ are indexed by the multi-indices $I$ with respect to the lexicographic order.

First, choose any $a_1$ columns among the first $i$ columns of the first block of \brref{canmat}, $a_2$ columns from the next $j_{12}$ columns, and $a_3$ columns from the last $j_{13}$ columns. Each such choice produces entries ${\mathcal T}^c_{I_{\underline a}, J,K}$ of the tensor ${\mathcal T}^c$, where $I_{\underline a}$ is any multi-index containing the indices of any non-chosen $i-a_1+j_{12}-a_2+j_{13}-a_3=s_1+1$ columns of the canonical form, and $J, K$ are any multi-indices of length, respectively, $s_2+1$ and $s_3+1$, as described above. Therefore, for any triplet $ \underline{a}=(a_1, a_2, a_3)$ the  $\bincof{i}{a_1} \bincof{j_{1,2}}{a_2} \bincof{j_{1,3}}{a_3}$ rows with indices $I_{\underline a}$ of the flattening ${\mathcal T}^c_1$ are zero.
\end{remark}



\section{Core of Grassmann Tensors}

\subsection{Core of Grassmann tensors in canonical form}

Let ${\mathcal T}$ be a trifocal Grassmann tensor and denote by ${\mathcal T}^c$ its canonical form. Results form the previous section allow one to directly find the core of ${\mathcal T}^c.$ This approach is similar to HOSVD (see Section \ref{hosvdgeneral}) but the canonical form of a tensor makes it easier to compute the matrices $U_j$ involved in the process. 

In Section \ref{multilin} we computed the multilinear rank $(r_1,r_2,r_3)$ of ${\mathcal T}^c$. As seen before, it is given by  $r_j = n_j - v_j$ where $v_j$ is the number of zero rows of $\T^c_j$. Moreover, Remark \ref{remfindvanishingrows} gives an effective method to list the zero rows $r_{h_1}, \dots, r_{h_{v_j}}$, in $\T^c_j$. Denote by $r_{k_1}, \dots, r_{k_{r_j}}$, with $r_{k_1} < r_{k_2} < \dots < r_{k_{r_j}}$ the non-zero rows of $\T^c_j$.

As remarked in the proof of Lemma \ref{lemmanminusv}, the columns of $\T^c_j$ are zero or they are elements of the canonical basis $\{\mathbf{e}_1, \dots ,\mathbf{e}_{n_j} \}$ of $\mathbb{C}^{n_j}$, i.e., among the columns of $\T^c_j$ we can find all vectors $\mathbf{e}_{k_t}$ for $t = 1,\dots r_j$. Hence it is straightforward to find an orthonormal basis for the image of each flattening and therefore the matrices $U_j$ quoted in \ref{hosvdgeneral}. Indeed, the matrix $U_j$ is the matrix whose columns are the vectors $\mathbf{e}_{k_1}, \dots, \mathbf{e}_{k_{r_j}}.$ Notice that deleting the zero rows $r_{h_1}, \dots, r_{h_{v_j}}$ from $U_j$ we get the identity matrix. What's more, the multiplication of $U^*_j$ by $\T^c_j$ deletes the zero rows of $\T^c_j$. As a consequence, the core tensor ${\mathcal C}^c$ of ${\mathcal T}^c$ is obtained from ${\mathcal T}^c$ by deleting all zero faces in each of the three directions.

\subsection{Core of Grassmann Tensors in the general case}
\label{core}

Let ${\mathcal T}$ be a trifocal Grassmann tensor and denote by ${\mathcal T}^c$ its canonical form. Recall that ${\mathcal T}^c$ can be obtained from ${\mathcal T}$ via multilinear multiplication, i.e., ${\mathcal T}^c=(V_1, V_2, V_3) \cdot {\mathcal T}$ where $V_i$ are invertible matrices obtained from the matrices $H_i$ and $K$ in Section \ref{canformsec}. More precisely, $V_j= (\bigwedge^{s_j+1}H_j^{-1})^*$ for $j=1,2,3$. As shown in the previous subsection, our construction of the canonical tensor allows us to introduce suitable matrices $U_1, U_2, U_3$ such that the core ${\mathcal C}^c$ of ${\mathcal T}^c$ can be obtained as follows: ${\mathcal C}^c=(U_1^*, U_2^*, U_3^*)\cdot {\mathcal T}^c$. 

In order to find a core ${\mathcal C}$ of ${\mathcal T}$, we proceed as follows. First, we define an invertible matrix $B_j$ of size $r_j \times r_j$ for $j=1,2,3$ as $B_j=E_jD_j^{-1}$ where $D_j$ is the diagonal matrix with the singular values of $V_j^{-1}U_j$ and $E_j$ is the matrix whose columns are the eigenvectors of $(V_j^{-1}U_j)^*(V_j^{-1}U_j)$.

Second, we define the tensor ${\mathcal C}$ as $(B_1^{-1}, B_2^{-1}, B_3^{-1}) \cdot {\mathcal C}^c$. Finally, we introduce matrices $S_j=V_j^{-1}U_jB_j$ for $j=1,2,3$, which are semi-orthogonal.

Third, we verify that ${\mathcal C}$ is a core of ${\mathcal T}$, i.e., ${\mathcal T}=(S_1, S_2, S_3)\cdot {\mathcal C}$ because ${\mathcal C}=(S_1^*, S_2^*, S_3^*)\cdot {\mathcal T}$ and the following diagram commutes:
\begin{equation}
\xymatrix{
{\mathcal T}  \ar[rr]^{(V_1, V_2, V_3)}  & & \ar[dd]^{(U_1^T, U_2^T, U_3^T)} {\mathcal T}^c \\
 & \\
{\mathcal C}  \ar[uu]^{(S_1, S_2, S_3)}  & & {\mathcal C}^c \ar[ll]^{(B^{-1}_1, B^{-1}_2, B^{-1}_3)}
}
\end{equation}

The matrices of the diagram above are computed in the following concrete example.

\subsubsection{Example}\label{coreconcreto} Let us consider $3$ projections from ${\mathbb P}^4$ onto, respectively, ${\mathbb P}^3$, ${\mathbb P}^2$, ${\mathbb P}^2$ having profile $(2,2,1)$ and corresponding to the following $3$ projection matrices $A_1, A_2, A_3$ such that
$$
A_1^T=\left(\begin{array}{cccc}
2 & 0 & 3 & 1 \\
0 & 0 & 1 & 0 \\
0 & 1 & 0 & 0 \\
1 & 0 & 0 & 0 \\
0 & 0 & 0 & 0 
\end{array}
\right), \quad 
A_2^T=\left(\begin{array}{ccc}
-1 & -1 & 1 \\
0 & 1 & 0 \\
0 & 0 & 0 \\
0 & 0 & 0 \\
1 & 0 & 0
\end{array}
\right), \quad 
A_3^T=\left(
\begin{array}{ccc}
0 & 0 & 0 \\
0 & 0 & 0 \\
1 & 0 & 0 \\
-1 & 1 & 0 \\
0 & 0 & 1
\end{array}
\right). 
$$

The associated $6 \times 3 \times 3$ Grassmann tensor ${\mathcal T}$ turns out to be
$$
\left[
\begin{array}{ccc|ccc|ccc}
0 & 0 & 0 & 0 & 0 & 0 & 0 & 0 & 0 \\
0 & 0 & 0 & 0 & 0 & 1 & 0 & 0 & 1 \\
0 & 0 & 0 & 0 & -1 & -4 & 0 & -1 & -4 \\
0 & 0 & 0 & 0 & 0 & 0 & 0 & 0 & -1 \\
0 & 0 & 0 & 0 & 0 & 0 & 0 & 1 & 4 \\
1 & 0 & 1 & 0 & 0 & 2 & 0 & 0 & 2 \\
\end{array}
\right]
$$

The matrices $(V_1, V_2, V_3)$ are given by
$$
V_1=
\left( 
\begin{array}{cccccc}
0 & 2 & 0& 0& 0 &-1 \\
2 & 0 & 0 & -3 & -1 & 0 \\
0 & -3 & -1& 0& 0 &0 \\
0 & 0 & 0& -1& 0 &0 \\
0 & -1 & 0& 0& 0 &0 \\
-1 & 0 & 0& 0& 0 &0 \\    
\end{array}
\right), \quad
$$
$$
V_2=
\left(
\begin{array}{ccc}
-1 & - 1& 1 \\
0 & 1 & 0 \\
1 & 0 & 0
\end{array}
\right), \quad 
V_3=
\left(
\begin{array}{ccc}
1 & 0 & 0 \\
0 & 1 & 0 \\
0 & -1 & 1
\end{array}
\right).
$$
The matrices $(U_1, U_2, U_3)$ are given by $(U_1, I_3, I_3)$ where
$$
U_1=
\left(
\begin{array}{ccccc}
1 & 0 & 0 &  0 & 0  \\
0 & 1 & 0 &  0 & 0  \\
0 & 0 & 1 &  0 & 0  \\
0 & 0 & 0 &  1 & 0  \\
0 & 0 & 0 &  0 & 1  \\
0 & 0 & 0 &  0 & 0  \\
\end{array}
\right)
$$

The tensor ${\mathcal T}^c$ has flattening ${\mathcal T}^c_3$ given by
$$
\left[
\begin{array}{ccc|ccc|ccc}
0 & 0 & -1 & 0 & 0 & 0 & 0 & 0 & 0 \\
0 & 0 & 0 & 0 & 0 & 0 & 0 & -1 & 0 \\
0 & 0 & 0 & 0 & 1 & 0 & 0 & 0 & 0 \\
0 & 0 & 0 & 0 & 0 & 0 & 1 & 0 & 0 \\
0 & 0 & 0 & -1 & 0 & 0 & 0 & 0 & 0 \\
0 & 0 & 0 & 0 & 0 & 0 & 0 & 0 & 0 \\
\end{array}
\right]
$$

Removing the last row in ${\mathcal T^c_3}$ gives the core ${\mathcal C}^c$ that is a $5 \times 3 \times 3$ tensor. For the convenience of the reader, we switch to another notation in order to describe ${\mathcal C}^c$ and ${\mathcal C}$ more clearly. Indeed, if we introduce canonical bases $\{a_i\}_{i=1}^5$, $\{b_j\}_{j=1}^3$, $\{c_k\}_{k=1}^3$ of the corresponding vector spaces, ${\mathcal C}^c$ is a sum of indecomposables as follows:
$$
{\mathcal C}^c= -a_1 \otimes b_3 \otimes c_1 + (a_3 \otimes b_2 - a_5 \otimes b_1) \otimes c_2 - (a_2 \otimes b_2 - a_4 \otimes b_1) \otimes c_3.
$$

The multilinear multiplication $(B_1^{-1}, B_2^{-1}, B_3^{-1})\cdot {\mathcal C}^c$ gives the core ${\mathcal C}$ of ${\mathcal T}$.
The matrices $(B_1^{-1}, B_2^{-1}, B_3^{-1})$ are given by
$$
B_1^{-1}=\left(
\begin{array}{ccccc}
0 & \frac{\sqrt{17}-\sqrt{13}}{2} & \frac{\sqrt{17}-\sqrt{13}}{2} & 0 & 0 \\
0 & \alpha & \beta & 0 & 0 \\
\frac{3}{2} &  0 & 0 & \frac{-13-\sqrt{221}}{13} & \frac{-13+ \sqrt{221}}{13} \\
\frac{\sqrt{13}+3}{2} & 0 & 0 & \gamma & \delta \\
0 & 0 & 0 & \frac{\sqrt{13}+\sqrt{17}}{2} & \frac{\sqrt{13}+\sqrt{17}}{2}
\end{array}
\right),
$$
where $$
\alpha=\frac{117\sqrt{3}-117\sqrt{3}-27\sqrt{39}+507}{676}, \quad \beta= \frac{117\sqrt{3}+117\sqrt{3}-27\sqrt{39}-507}{676}, 
$$
$$ \gamma=\frac{39\sqrt{13}+39\sqrt{17}+9\sqrt{221}+117}{52}, \quad \delta = \frac{39\sqrt{13}-39\sqrt{17}-9\sqrt{221}+117}{52},
$$
and 
$$
B_2^{-1}= \left(
\begin{array}{ccc}
\frac{\sqrt{2}-\sqrt{6}}{4} & \frac{\sqrt{6}}{2} & \frac{\sqrt{6}}{2} \\
0 & -\sqrt{3} & \sqrt{3} \\
\frac{\sqrt{2}+\sqrt{6}}{4} & \frac{3\sqrt{2}+2\sqrt{6}}{2}  & \frac{3\sqrt{2}+2\sqrt{6}}{2}
\end{array}
\right), \quad
B_3^{-1}=\left( 
\begin{array}{ccc}
0 & \frac{9\sqrt{5}-25}{22} & \frac{15-\sqrt{5}}{11} \\
1 & 0 & 0 \\
0 & \frac{8\sqrt{5}-10}{11} & \frac{8 \sqrt{5}-10}{11}
\end{array}
\right).
$$

If we denote by $K^i_{j_i}$ the $j_i$-th column of $B^{-1}_i$ for $1 \leq i \leq 3$, $1 \leq j_1 \leq 5$, $1 \leq j_2 \leq 3$ and $1 \leq j_3 \leq 3$, the core of ${\mathcal T}$ can be written as 
 
$$
-K^1_1 \otimes K^2_3 \otimes K^3_1 + (K^1_3 \otimes K^2_2 - K^1_5 \otimes K^2_1) \otimes K^3_2- (K^1_2 \otimes K^2_2 - K^1_4 \otimes K^2_1) \otimes K^3_3.
$$

\section{Examples}\label{examples}

In this section we provide seven different examples of three projections $P_j:\Pin{k}\mathrel{-\,}\rightarrow \Pin{h_j}$ $j=1,2,3,$ leading to trifocal Grassman tensors of dimension $n_1 \times n_2 \times n_3,$ with profile $(\alpha_1, \alpha_2, \alpha_3),$ whose multilinear rank is explicitly computed following Lemma \ref{lemmanminusv}, Lemma \ref{lemmajrsconditions}, and Proposition \ref{Frank}. In the first example we also explicitly identify the zero rows responsible for the drop in rank of the first flattening, and we also provide explicit matrices needed for the calculation of the core. Recall that $k$ is determined by the profile, i.e., $k=\alpha_1+\alpha_2+\alpha_3-1,$ $i$ and $j_{rs}$ are defined in \brref{numrel2} and \brref{numrel3}, and $m_1,m_2,$ $A,$ $|A|,$ and $N$ are defined in Claim \ref{remsetA}.\\

\begin{landscape}
\begin{tabular}{|c|c|c|c|c|c|c|c|}
\hline
 & {\it Example 1} &  {\it Example 2} & {\it Example 3} & {\it Example 4} & {\it Example 5}& {\it Example 6}& {\it Example 7}\\
\hline
$k$ & $7$ & $6$ & $9$ & $5$ & $8$ & $12$ & $9$\\\hline
$(h_1, h_2, h_3)$ & $(6,4,4)$ & $(5,4,3)$ & $(8,6,4)$ & $(2,4,4)$ & $(5,5,5)$ & $(7,8,8)$& $(3,8,8)$\\
\hline
$(\alpha_1, \alpha_2, \alpha_3)$ & $(3,3,2)$ & $(3,2,2)$ & $(4,3,3)$ & $(2,2,2)$ & $(1,4,4)$ & $(1,6,6)$ & $(3,3,4)$\\
\hline
$i$ & $1$ & $1$ & $1$ & $1$ & $0$ & $0$& $2$ \\
\hline
 $(j_{1,2}, j_{1,3}, j_{2,3})$ & $(3,3,1)$ & $(3,2,1)$ & $(5,3,1)$ &  $(1,1,3)$ & $(3,3,3)$ & $(4,4,5)$ & $(1,1,6)$\\
 \hline
 $(n_1, n_2, n_3)$&$(35,10,10)$&$(20,10,6)$&$(126,35,10)$& $(3,10,10)$ & $(6,15,15)$& $(8,84,84)$ & $(4,84,126)$\\
 \hline
 F-$\rk{\T}$&$$(31,10,10)$$ &$(19,10,6)$ &$(105,35,10)$ & $(3,9,9)$& $(6,12,12)$ & $(8,50,50)$ & $(4,65,75)$\\
 \hline
\end{tabular}
\end{landscape}

\begin{example} In this case we consider three projections from $\Pin{7}$ to, respectively, $\Pin{6}$, $\Pin{4},$ and $\Pin{4},$ with profile $(3,3,2).$ $\T$ is a tensor of dimension $35 \times 10 \times 10$ and the value of the quantities in \brref{numrel2} and \brref{numrel3} are as in the table above. Notice that Assumption \ref{g.a.} is satisfied as $i = 1$ and hence the construction of $\T^c$ can be performed. The only values of $r,s$ that satisfy one of \brref{jrineq} are $r=1,s=3,$ as $j_{1,3}-(\alpha_3+1) =0.$ Hence $\rk{\T^c_2}$ and $\rk{\T^c_3}$ are both maximal while $\rk{\T^c_1}$ drops. For $(r,s,t) = (1,3,2)$ Claim \ref{remsetA} shows that $A = \{(0,3,0), (1,2,0)\}$. According to Proposition \ref{Frank}, the contribution to the rank deficiency given by the first triplet in $A$ is $\binom{1}{0}\binom{3}{3}\binom{3}{0}=1$ while the contribution of the second triplet is $\binom{1}{1}\binom{3}{2}\binom{3}{0}=3$. Therefore, $\rk{\T_1^c}$ drops by $4,$ and F-$\rk{\T} = (31,10,10).$

Following Remark \ref{remfindvanishingrows} we will now identify the $4$ zero rows of $\T^c_1.$ Notice that the first block of \brref{canmat} is a submatrix of dimension $9 \times 7$ while $\alpha_1 = 3,$ hence the multi-indices of the sets of columns that are not chosen in the calculation of each maximal minor, i.e. the row multi-indices of $\T_1^c,$ have length $4$ are the following, in proper lexicographic order, listed above their corresponding row indices:

$$\begin{tabular}{ccccccccc}
1234 & 1235 & 1236& ...&...&...&3467&3567& 4567\\
1&2&3&...&...&...&33&34&35.\\
\end{tabular}$$

More specifically, the entries of the first row of ${\mathcal T}_1$ are given by ${\mathcal T}_{1234, J, K}$, where $J$ and $K$ are multi-indices as in Section \ref{canformsec}.

In correspondence of the first triplet $(0,3,0)$, as $a_2 =j_{12} = 3,$ we are forced to choose the second, the third and the fourth column of the first block of \brref{canmat} to compute entries of the tensor. These entries ${\mathcal T^c}_{I, J, K}$ correspond to the multi-index $ I=\{1567\},$ i.e. the $20$th row of ${\mathcal T}^c_1$.

On the other hand, the triplet $(1,2,0)$ gives three possible multi-indices of rows. As $i=a_1 =1$ we are forced to choose the first column of the first block of \brref{canmat}. As $j_{12} = 3$ and $a_2 = 2$, we have three possible choices $\{(2,3), (2,4), (3,4)\}$ for two out of the next three columns of the same block. Hence we have, respectively, the three rows $ {\T^c_1}_{4567, J, K}, {\T^c_1}_{3567, J, K},{\T^c_1}_{2567, J, K},$ i.e. rows $30$, $34$, and $35.$

\end{example}

\begin{example}
In this case we consider three projections from $\Pin{6}$ to, respectively, $\Pin{5}$, $\Pin{4},$ and $\Pin{3},$ with profile $(3,2,2).$ $\T$ is a tensor of dimension $20 \times 10 \times 4$ and the value of the quantities in \brref{numrel2} and \brref{numrel3} are as in the table above. Notice that Assumption \ref{g.a.} is satisfied as $i = 1$ and hence the construction of $\T^c$ can be performed. The only values of $r,s$ that satisfy one of \brref{jrineq} are $r=1,s=2,$ as $j_{1,2}-(\alpha_2+1) =0.$ Hence $\rk{\T_2}$ and $\rk{\T_3}$ are both maximal while $\rk{\T_1}$ drops. For $(r,s,t) = (1,2,3)$ Claim \ref{remsetA} shows that $A = \{(1,0,2)\}$. According to Proposition \ref{Frank}, the contribution to the rank deficiency given by the triplet in $A$ is $\binom{1}{1}\binom{3}{0}\binom{2}{2}=1$ therefore $\rk{T^1_c}$ drops by $1$ and F-$\rk{\T} = (19,10,4).$

\end{example}

\begin{example} In this case we consider three projections from $\Pin{9}$ to, respectively, $\Pin{8}$, $\Pin{6},$ and $\Pin{4},$ with profile $(4,3,3).$ $\T$ is a tensor of dimension $126 \times 35 \times 10$ and the value of the quantities in \brref{numrel2} and \brref{numrel3} are as in the table above. Notice that Assumption \ref{g.a.} is satisfied as $i = 1$ and hence the construction of $\T^c$ can be performed.  In this case, only one set of values satisfy one of \brref{jrineq}, $(r,s,t) = (1,2,3)$, as $j_{1,2}-(\alpha_2+1) = 1 \geq \max(0,\alpha_1 - i - j_{1,3}) =0.$
 Hence $\rk{\T_1}$ drops, while both $\rk{\T_2}$  and $\rk{\T_3}$ are maximal. For $(r,s,t) = (1,2,3)$ Claim \ref{remsetA} shows that $A = \{(1,0,3), (0,1,3), (1,1,2)\}$. According to Proposition \ref{Frank}, the contribution to the rank deficiency given by the first triplet in $A$ is $\binom{1}{1}\binom{5}{0}\binom{3}{3}=1;$ the contribution of the second triplet is $\binom{1}{0}\binom{5}{1}\binom{3}{3}=5;$ and the contribution of the third triplet is $\binom{1}{1}\binom{5}{1}\binom{3}{2}=15.$ Therefore $\rk{T^1_c}$ drops by $21.$ 
\end{example}
\begin{example} 
 In this case we consider three projections from $\Pin{5}$ to, respectively, $\Pin{2}$, $\Pin{4},$ and $\Pin{4},$ with profile $(2,2,2).$ $\T$ is a tensor of dimension $3 \times 10 \times 10$ and the value of the quantities in \brref{numrel2} and \brref{numrel3} are as in the table above. Notice that Assumption \ref{g.a.} is satisfied as $i = 1$ and hence the construction of $\T^c$ can be performed. The only values of $r,s$ that satisfy one of \brref{jrineq} are $r=2,s=3,$ as $j_{2,3}-(\alpha_3+1) = 0$ and $r=3,s=2,$ as $j_{2,3}-(\alpha_2+1) = 0.$ Hence $\rk{\T_1}$ is maximal, while both $\rk{\T_2}$ and $\rk{\T_3}$ drop. For $(r,s,t) = (2,3,1)$ Claim \ref{remsetA} shows that $B = \{(1,1,0)\}$. According to Proposition \ref{Frank}, the contribution to the rank deficiency of $\rk{T^2_c}$ is given by the triplet in $B$ which is $\binom{1}{1}\binom{1}{1}\binom{3}{0}=1$. Therefore $\rk{T^2_c}$ drops by $1$. Similarly the contribution to the rank deficiency of $\rk{T^3_c}$ is given by the triplet in $C= \{(1,1,0)\}$ which is $\binom{1}{1}\binom{1}{1}\binom{3}{0}=1$. Therefore $\rk{T^3_c}$ drops by $1$ too and  F-$\rk{\T} = (3,9,9).$ 
 \end{example}
 \begin{example} 
 In this case we consider three projections from $\Pin{8}$ to $\Pin{5}$ with profile $(1,4,4).$ $\T$ is a tensor of dimension $6 \times 15 \times 15$ and the value of the quantities in \brref{numrel2} and \brref{numrel3} are as in the table above. Notice that Assumption \ref{g.a.} is satisfied as $i = 0$ and hence the construction of $\T^c$ can be performed. The two sets of values of $r,s$ that satisfy one of \brref{jrineq} are $r=2,s=1,$ as $j_{2,1}-(\alpha_1+1) = 1 \geq \max(0,\alpha_2 - i - j_{2,3})=0 $ and $r=3,s=1,$ as $j_{3,1}-(\alpha_1+1) = 1 \geq \max(0,\alpha_1 - i - j_{3,2})=0 .$ Hence $\rk{\T_1}$ is maximal, while both $\rk{\T_2}$ and $\rk{\T_3}$ drop. Proceeding as in previous examples in this section one gets  F-$\rk{\T} = (6,12,12).$ 
 \end{example}
 \begin{example} 
 In this case we consider three projections from $\Pin{12}$ to, respectively, $\Pin{7}$, $\Pin{8},$ and $\Pin{8},$ with profile $(1,6,6).$ $\T$ is a tensor of dimension $8 \times 84 \times 84$ and the value of the quantities in \brref{numrel2} and \brref{numrel3} are as in the table above. Notice that Assumption \ref{g.a.} is satisfied as $i = 0$ and hence the construction of $\T^c$ can be performed. The only values of $r,s$ that satisfy one of \brref{jrineq} are $r=2,s=1$ and $r=3,s=1,$as we have the strict inequalities: $j_{2,1}-(\alpha_1+1) = 2 > \max(0,\alpha_2 - i - j_{2,3})=1 $ and $j_{3,1}-(\alpha_1+1) = 2 > \max(0,\alpha_1 - i - j_{3,2})=1 .$ Hence $\rk{\T_1}$ is maximal, while both $\rk{\T_2}$ and $\rk{\T_3}$ drop. Proceeding as in previous examples in this section one gets F-$\rk{\T} = (8,50,50).$ 
 \end{example}
 \begin{example} 
 In this last case we have an example of tensor with a relatively small core. We consider three projections from $\Pin{9}$ to $\Pin{3}$, $\Pin{8},$ and $\Pin{8},$  with profile $(3,4,4).$ $\T$ is a tensor of dimension $4 \times 84 \times 126$, $i = 1.$ Proceeding as in previous examples in this section one gets F-$\rk{\T} = (4,65,75).$ 
 \end{example}

\end{document}